\documentclass[12pt]{article}
\usepackage{e-jc}

\usepackage{amsthm,amsmath,amssymb}

\usepackage{graphicx}
\usepackage{tikz}

\usepackage{enumerate}

\usepackage[colorlinks=true,citecolor=black,linkcolor=black,urlcolor=blue]{hyperref}

\theoremstyle{plain}
\newtheorem{theorem}{Theorem}

\newtheorem{claim}[theorem]{Claim}

\theoremstyle{definition}

\theoremstyle{remark}
\newtheorem{remark}[theorem]{Remark}

\newcommand{\B}{\mathcal{B}}
\newcommand{\cu}{C_{\rm up}}
\newcommand{\cd}{C_{\rm down}}
\newcommand{\NN}{\mathcal{N}}
\newcommand{\La}{{\rm La}}
\newcommand{\F}{{\mathcal F}}

\newcommand{\symmdiff}{\bigtriangleup}

\title{\bf A note on the size of $\NN$-free families}

\author{Ryan R. Martin\thanks{This work was supported by a grant from the Simons Foundation (\#353292, Ryan R. Martin).}\\
\small Department of Mathematics \\[-0.8ex]
\small Iowa State University\\[-0.8ex] 
\small 396 Carver Hall\\[-0.8ex] 
\small Ames, Iowa, U.S.A.\\
\small\tt rymartin@iastate.edu\\
\and
 Shanise Walker \\
\small Department of Mathematics \\[-0.8ex]
\small Iowa State University\\[-0.8ex] 
\small 396 Carver Hall\\[-0.8ex] 
\small Ames, Iowa, U.S.A.\\
\small\tt shanise1@iastate.edu\\
}

\date{\dateline{February 21, 2017}{}}

\begin{document}

\maketitle

\begin{abstract}
  The $\NN$ poset consists of  four distinct sets $W,X,Y,Z$ such that $W\subset X$, $Y\subset X$, and $Y\subset Z$ where $W$ is not necessarily a subset of $Z$. A family $\F$, considered as a subposet of the $n$-dimensional Boolean lattice  $\B_n$,  is $\NN$-free if it does not contain $\NN$ as a subposet.   Let $\La(n, \NN)$ be the size of a largest $\NN$-free family  in $\B_n$. Katona and Tarj\'{a}n proved that $\La(n,\NN)\geq {n \choose k}+A(n,4,k+1)$, where $k=\lfloor n/2\rfloor$ and $A(n, 4, k+1)$ is the size of a single-error-correcting code with constant weight $k+1$. In this note, we prove for $n$ even and $k=n/2$, $\La(n, \NN) \geq {n\choose k}+A(n, 4, k)$, which improves the bound on $\La(n, \NN)$ in the second order term for some values of $n$ and should be an improvement for an infinite family of values of $n$, depending on the behavior of the function $A(n,4,\cdot)$.

\end{abstract}

\section{Introduction}

The \emph{$n$-dimensional Boolean lattice}, $\B_n$, denotes the partially ordered set (poset) $(2^{[n]}, \subseteq)$, where $[n]=\{1,\ldots,n\}$ and, for every finite set $S$, $2^S$ denotes the set of subsets of $S$. 
For posets, $P=(P, \preceq)$ and $P^{\prime}=(P^{\prime}, \preceq)$, we say $P^{\prime}$ is a  \emph{(weak) subposet} of $P$ if there exists an injection $f:P^{\prime}\rightarrow P$ that preserves the partial ordering. That is, whenever $u\leq^{\prime} v$ in $P^{\prime}$, we have $f(u)\leq f(v)$ in $P$.  If $\F$ is a subposet of $\B_n$ such that $\F$ contains no subposet $P$, we say $\F$ is \emph{$P$-free}. 

$P$-free posets (or $P$-free families) have been extensively studied, beginning with Sperner's theorem in 1928. Sperner~\cite{Sperner} proved that the size of the largest antichain in $\B_n$ is ${n \choose {\lfloor n/2\rfloor}}$. Erd\H{o}s~\cite{Erdos} generalized this result to chains. Katona and Tarj\'{a}n~\cite{KatonaTarjan} addressed the problem of ${\mathcal V}$-free families and got an asymptotic result.  Griggs and Katona~\cite{GriggsKatona} addressed $\NN$-free families, obtaining Theorem~\ref{thm:CD:JGGK} below. See Griggs and Li~\cite{GriggsLi} for a survey of the progress on $P$-free families.  Let $\La(n, P)$ denote the size of the largest $P$-free family in $\B_n$.

The main result of this note is Theorem~\ref{thm:CD:lowerboundnfree}, in which, for some values of $n$, we improve the bounds on $\La(n,\NN)$ in the second-order term. The poset $\NN$ consists of four distinct sets $W,X,Y,Z$ such that $W\subset X$, $Y\subset X$, and $Y\subset Z$. However, $W$ is not necessarily a subset of $Z$.  See Figure~\ref{fig:npic}. The earliest extremal result on $\NN$-free families is Theorem~\ref{thm:CD:JGGK}.

\tikzstyle{vertex}=[circle, draw, inner sep=0pt, minimum size=6pt]
\tikzstyle{rvertex}=[circle, red, fill, draw, inner sep=0pt, minimum size=6pt]
\tikzstyle{gvertex}=[circle, green, fill, draw, inner sep=0pt, minimum size=6pt]
\tikzstyle{bvertex}=[circle, blue, fill, draw, inner sep=0pt, minimum size=6pt]
\tikzstyle{Bvertex}=[circle, black, fill, draw, inner sep=0pt, minimum size=6pt]
\tikzstyle{pvertex}=[circle, purp, fill, draw, inner sep=0pt, minimum size=6pt]
\tikzstyle{overtex}=[circle, orange, fill, draw, inner sep=0pt, minimum size=6pt]
\newcommand{\vertex}{\node[vertex]}
\newcommand{\rvertex}{\node[rvertex]}
\newcommand{\gvertex}{\node[gvertex]}
\newcommand{\bvertex}{\node[bvertex]}
\newcommand{\Bvertex}{\node[Bvertex]}
\newcommand{\pvertex}{\node[pvertex]}

\begin{figure}[ht]\center
	\begin{tikzpicture}[scale=1.2]
		\Bvertex (1) at (0,0) [label=left:W]{};
		\Bvertex (2) at (0, 1) [label=left:X]{};
		\Bvertex (3) at (1, 0) [label=right:Y]{};
		\Bvertex (4) at (1,1) [label=right:Z]{};
		\draw (1) to (2);
		\draw (2) to (3);
		\draw (3) to (4);
	\end{tikzpicture}
	\caption{The $\NN$ poset.} \label{fig:npic}
\end{figure}

\begin{theorem}[Griggs and Katona~\cite{GriggsKatona}]\label{thm:CD:JGGK}
\[ {n \choose {\lfloor n/2\rfloor}}\left(1+\frac{1}{n}+\Omega\left(\frac{1}{n^2}\right)\right)\leq \La(n,\NN)\leq {n \choose {\lfloor n/2\rfloor}}\left(1+\frac{2}{n}+O\left(\frac{1}{n^2}\right)\right).\]

\end{theorem}

The construction for the lower bound of Theorem~\ref{thm:CD:JGGK} comes directly from a previous result of Katona and Tarj\'{a}n~\cite{KatonaTarjan} from 1983 on $\mathcal V$-free families. The poset $\mathcal V$ consists of three elements $X,Y,Z$ such that $Y\subset X$ and $Y\subset Z$. It is clear that $\La(n,{\mathcal V})\leq\La(n,\NN)$ because any $\mathcal V$-free family is also $\NN$-free.

To establish the lower bound, Katona and Tarj\'an used a constant-weight code construction due to Graham and Sloane~\cite{GrahamSloane} from 1980.  In the proof of Theorem~\ref{thm:CD:lowerboundnfree}, we obtain a lower bound that appears to be larger than the current known bound. However, whether it is an improvement depends on the behavior of some functions well-known in coding theory. In order to discuss our results we need some brief coding theory background. 

\subsection{Coding Theory Background}
Let $A(n,2\delta,k)$ denote the size of the largest family of $\{0,1\}$-vectors of length $n$ such that each vector has exactly $k$ ones and the Hamming distance between any pair of distinct vectors is at least $2\delta$.  This is the same as the size of the largest family of subsets of $[n]$ such that each subset has size exactly $k$ and the symmetric difference of any pair of distinct sets is at least $2\delta$.

The quantity $A(n,2\delta,k)$ is important in the field of error-correcting codes. In fact, $A(n,4,k)$ computes the size of a single-error-correcting code with constant weight $k$. Henceforth, we will use ``SEC code'' as shorthand for ``single-error-correcting code.''

The first nontrivial value of $\delta$ for $A(n,2\delta,k)$ is $\delta=2$. Graham and Sloane~\cite{GrahamSloane} give a lower bound construction for $A(n, 4, k)$. 

\begin{theorem} [Graham and Sloane~\cite{GrahamSloane}]\label{thm:CD:RGNS} 
	$A(n,4,k)\geq \frac{1}{n}{n \choose k}.$
\end{theorem}

\subsection{Main Result}

Katona and Tarj\'{a}n~\cite{KatonaTarjan}  estimated the following lower bound for $\NN$-free families. 

\begin{theorem}\label{thm:CD:lowerboundKT} Let $k=\lfloor n/2\rfloor$. Then,
	\begin{align*}
		\La(n,\NN)\geq {n \choose k}+A(n,4,k+1) .
	\end{align*}
\end{theorem}

The following theorem is our main result of the note. 
\begin{theorem}\label{thm:CD:lowerboundnfree} Let $n$ be even and let $k=n/2$. Then,
	\begin{align}
		\La(n, \NN) \geq {n\choose k}+A(n, 4, k) . \label{eq:secondbound}
	\end{align}

\end{theorem}

\begin{remark} 
	This is potentially an improvement when $n$  is even. We note that the same 3-level construction works for $n$ odd and $k=(n-1)/2$. This gives a family of size ${n\choose k}+A(n,4,k)$ nontrivially in three layers. However, since $A(n,4,k)=A(n,4,k+1)$ in the odd case, this does not provide an improvement to the known bounds.
	
	We believe that, for $n\geq 6$, the quantity $A(n,4,k)$ is strictly unimodal as a function of $k$ as long as $3\leq k\leq n-3$. This strict unimodality has been established~\cite{BSSS} for $6\leq n\leq 12$ and known bounds suggest that it is the case for larger values of $n$ as well. If unimodality holds, then $A(n,4,k)$ would achieve its maximum uniquely at $k=\lfloor n/2\rfloor$ or $k=\lceil n/2\rceil$. Therefore, we expect (\ref{eq:secondbound}) to also be a strict improvement over Theorem~\ref{thm:CD:lowerboundKT} in the case where $n$ is even. However, to our knowledge, the unimodality of $A(n,4,k)$ has never been established and seems to be a highly nontrivial problem.

\end{remark}

\noindent{\bf{Proof of Theorem 4.}}

Given $k=n/2$, let $C$ be a constant weight SEC code of size $A(n,4,k)$. Define $\cu:=\{c\cup\{i\}: c\in C, \hspace{2px}  i\notin c\}$ and $\cd:=\{c-\{i\}: c\in C, \hspace{2px} i\in c\}$. Claim~\ref{cl:CD:cupcdown} gives some important properties of $\cu\cup\cd$.

\begin{claim}\label{cl:CD:cupcdown}~
	\begin{enumerate}[(i)]
		\item Both $\cu$ and $\cd$ are SEC codes with constant weight $k+1$ and $k-1$, respectively. \label{it:cucdcodes}
		\item If $c'' \in \cu$ and $c' \in \cd$, $c' \not\subseteq c''$. \label{it:cdnotincu}	
		\end{enumerate}
\end{claim}

\begin{proof}
	{\bf\it (\ref{it:cucdcodes}).} Let $c_1, c_2\in \cu$. Then $\left|c_1\symmdiff c_2\right|=\left|(c_1-\{i\})\symmdiff (c_2-\{i\})\right|\geq 4$ since $(c_1-\{i\}),  (c_2-\{i\})\in C$ and their symmetric difference must be at least $4$ in order for $C$ to be a $1$-EC code. Thus, $\cu$ is a SEC code. By a similar argument, $\cd$ is a SEC code. 
	
	{\bf\it (\ref{it:cdnotincu}).} Let $c'' \in \cu$,  $c' \in \cd$, and $c' \subset c''$. Then, $(c' \cup \{i\}), (c'' - \{i\}) \in C$. So, $\left|(c'' - \{i\})\symmdiff (c' \cup \{i\})\right|\geq 4$. This implies that there are two members of $[n]$ that are in $(c' \cup \{i\})-(c'' - \{i\})$. One is $i$ and the other is some $j\in c'-c''$, which contradicts the assumption that $c'\subset c''$. This concludes the proof of Claim~\ref{cl:CD:cupcdown}.
\end{proof}

In order to finish the proof, we just need to show that the family ${\mathcal F}:={[n] \choose k}\cup \cu \cup \cd$ is $\NN$-free. 

To that end, suppose there is a subposet $\NN$ with elements $W,X,Y,Z$ where $W\subset X$, $Y\subset X$ and $Y\subset Z$ (see Figure~\ref{fig:npic}). Where is the element $X$?

We know that $X \not\in \cd$ because it has to have elements below it and the elements of $\cd$ are all minimal in $\mathcal{F}$. We know that $X\not\in {[n]\choose k}$ because that would force $W,Y \in \cd$ and, being subsets of $X$ would require $|W\symmdiff Y|=2$, a contradiction to $\cd$ being a SEC code. Therefore, $X \in \cu$.

Now, where is $Y$? We know that $Y \not\in \cu$ because $Y\subset X$. We know $Y\not\in {[n]\choose k}$ because that would force $X,Z \in \cu$ and thus would force $|X\symmdiff Z|=2$, this is a contradiction to the fact that $\cu$ is a SEC code. Therefore, $Y \in \cd$.

In order for the copy of $\NN$ to exist, $Y\subset X$, which implies $Y\subset X-\{i\}$ and so $\left|(Y\cup\{i\})\symmdiff (X-\{i\})\right|=2$. Recall, however, that $Y\cup\{i\}$ and $X-\{i\}$ are distinct members of $C$ and so have symmetric difference at least 4, a contradiction.
\qed

\subsection*{Acknowledgements}
We would like to extend our thanks to Kirsten Hogenson and Sung-Yell Song for providing helpful conversations.


\end{document}